\documentclass[12pt]{article} % use larger type; default would be 10pt

\usepackage[utf8]{inputenc} % set input encoding (not needed with XeLaTeX)
\usepackage{amsmath, amssymb, amsthm}
\usepackage{float}
\newtheorem{thm}{Theorem}[section]

\def\lc{\left\{}
\def\rc{\right\}}
\def\lr{\left(}
\def\rr{\right)}
\usepackage{geometry} % to change the page dimensions
\usepackage{graphicx, tikz} % support the \includegraphics command and options
\usetikzlibrary{arrows,trees,matrix,shapes.geometric}
\usepackage{booktabs} % for much better looking tables
\usepackage{array} % for better arrays (eg matrices) in maths
\usepackage{paralist} % very flexible & customisable lists (eg. enumerate/itemize, etc.)
\usepackage{verbatim} % adds environment for commenting out blocks of text & for better verbatim
\usepackage{subfig} % make it possible to include more than one captioned figure/table in a single float
\usepackage{fancyhdr} % This should be set AFTER setting up the page geometry
\usepackage{sectsty}
\allsectionsfont{\sffamily\mdseries\upshape} % (See the fntguide.pdf for font help)
\usepackage[nottoc,notlof,notlot]{tocbibind} % Put the bibliography in the ToC
\usepackage[titles,subfigure]{tocloft} % Alter the style of the Table of Contents

\usepackage[english]{babel}
\usepackage{tablefootnote}

\title{Universal and Overlap Cycles for Posets, Words, and Juggling Patterns} 
\author{Adam King, University of Louisville\and Amanda Laubmeier, University of Arizona\and Kai Orans, Pomona College\and Anant Godbole, East Tennessee State University}
\begin{document}
\maketitle

\begin{abstract} %JMM abstract should go here?
We discuss results dealing with universal cycles (u-cycles) and $s$-overlap cycles, and contribute to the body of those results by proving existence of universal cycles of naturally labeled posets (NL posets), $s$-overlap cycles of words of weight $k$, and juggling patterns. The result on posets is, to the best of our knowledge, the first demonstration of the existence of a u-cycle whose length is unknown.
\end{abstract}

%we should standardize capitalization for u/o cycles, but I don't actually know which one you want to use.
%I commented out duplicate explanations

\section{Introduction}~

A universal cycle, or u-cycle, is an encoding of a set of combinatorial objects as a cycle of the basic elements constituting those objects.  For instance, a u-cycle of a set of graphs will be constructed from vertices and edges, while a u-cycle of a set of labeled posets will be constructed from labeled vertices and directed edges.  Most often the objects we attempt to build a u-cycle of are strings of integers, and in this case the u-cycle is a cycle of integers.  Each object is represented by $k$ adjacent elements in the cycle -- a $k$-window, for some $k\ge2$.  These $k$-windows overlap each other; in fact, each such window shares $k-1$ elements with the window before it and $k-1$ elements with the window after it.  Thus a u-cycle of a set of $n$ objects will have $n$ overlapping $k$-windows.  The length of a u-cycle is the number of objects it encodes.

%Universal cycles ($u$-cycles) are sequences that encode objects. Each object is represented by a window of length $k$ on a string of length $n$. All flavors of the object are uniquely represented by one such $k$-window, shifting one character at a time.

An example of a u-cycle is the cyclic string $11010001$, which cycles through all binary words of length 3. Here $k=3$, $n=8$, and the first object represented is the word $110$. Shifting one character, the window represents the word $101$. This process continues, cycling back to the beginning of the string, until all 8 words appear in the sequence, once and only once each.

An $s$-overlap cycle, $s$-ocycle, or simply ocycle, is a generalization of a u-cycle.  Each object is represented by a $k$-window, but the overlap between the windows is of magnitude $s\le k-2$ instead of $k-1$.  Because of this, an $s$-ocycle of a set of $n$ objects will have length $(k-s)\cdot n$.  Given that u-cycles and $s$-ocycles are primarily a way to condense information, we most often construct $(k-2)$-ocycles when u-cycles don't exist.  However, if $\gcd(n,s) \neq 1$ then an $s$-ocycle might not be possible \cite{horan}, \cite{ocyclesHoran}.  Thus in our results we prove existence of $s$-ocycles for all (or some) $s$ such that $\gcd(n,s) = 1$.  The optimal $s$-ocycle has the largest such $s$. Ocycles were first introduced in \cite{anant} and systematically studied in \cite{hh2} and \cite{ocyclesHoran}.  See also \cite{horan} where Horan studies necessary and sufficient conditions for the values of $s$ that admit ocycles.% 'expanded upon' replaced 'systematically exploited' -- I'm not 100% okay with this but I can't figure out why...I am OK with it sez Anant

%Universal cycles can be generalized to $s$-overlap cycles ($o$-cycles) by increasing the amount shifted between objects. This allows for more flexibility in cycle construction, and some objects which cannot be $u$-cycled can be $o$-cycled.

%Each combinatorial object represented in a $u$-cycle shares $k-1$ characters with the next representation in the cycle.  Each combinatorial object represented in an $s$-overlap cycle shares $s$ characters with the next representation in the cycle.  So a $u$-cycle is a $k-1$-overlap cycle. 

In what follows we will often use ``u-cycle" and ``ocycle" as verbs, much like ``google", and make statements such as ``Some objects which cannot be u-cycled can be ocycled." For example, permutations on 3 letters cannot be represented by a u-cycle. However, the ocycle \begin{center} $123213213123$ \end{center} encodes all permutations on 3 letters. This string is a $1$-overlap cycle with length-3 windows. The first permutation is $123$ and a two-character shift gives the second permutation $321$. This process continues until all 6 permutations appear in the string.  See \cite{ocyclesHoran} for general results on ocycles for permutations.

%Kai's table should go here. expect ungodly white space to follow?

%Anant said to use Antonio here, but I think the graph paper fits, too.

Papers of primary relevance to us are those by Chung, Diaconis and Graham \cite{granddaddyCycles}, Chung and Graham \cite{jugglingChung}, Blanca and Godbole \cite{cyclesAntonio}, Campbell, Godbole and Kay \cite{cyclesAndre}, Horan and Hurlbert \cite{ocyclesHoran}, and Horan \cite{horan}.  The relationships between these are summarized in Figure 1.  A detailed description  follows:  The parent paper to all present and past work on u-cycles is the landmark \cite{granddaddyCycles}.  Some structures examined there were $k$-subsets of $[n]:=\{1,2,\ldots,n\}$ with some permutation of the set $A=\{a_1,\ldots,a_k\}$, written sequentially, representing the set $A$; for example the sequence 1234524135 is a u-cycle of all the 2-subsets of $\{1,2,3,4,5\}$.  (Notice that we have abused the notation somewhat here; the ``$n$" of the general discussion above equals, in this context, ${n\choose k}$.)  We call this the $k$-coding and important progress in this regard was made in the paper of Hurlbert \cite{hurlbert}.  Changing the coding so that a $k$-set is represented by its $n$-long characteristic vector, it was proved in \cite{cyclesAntonio} that a u-cycle could be created of all subsets of sizes in the range $[s,t];t>s$; an example with $s=2, t=3, n=4$ is given by 1110011010.  We call this the $n$-coding.  The authors of \cite{cyclesAntonio} also proved that words of weight in a suitably restricted range, and over an alphabet of size $d$, could be u-cycled, thus extending the subsets result to multisets.  This result was made less restrictive in \cite{cyclesAndre}.  Of course there is no way to produce a u-cycle for subsets of fixed size, or words of fixed weight, using either the $n$-coding or the $k$-coding, and the \$100 conjecture of \cite{granddaddyCycles}, namely that $k$-coded u-cycles of $k$-subsets of $[n]$ exist iff $n\ge n_0(k)$ and $n\vert{n\choose k}$ remains tantalizingly open.  A different kind of breakthrough appeared in \cite{hh2}, where the authors proved that the set of permutations of a fixed size $n$ multiset could be $s$-ocycled if the right divisibility condition was satisfied.  Here is a fact that was not explicitly mentioned there:  Given a binary alphabet and the fixed multiset $A=\{0,0,\ldots,0,1,1,\ldots,1\}$ with $k$ ones and $n-k$ zeros, we note that the set of permutations of $A$ is simply the $n$-coding of all $k$-subsets of $[n]$, so that these notoriously difficult objects {\it can} be ocycled in the binary $n$-coding.  In one of the main results of this paper, proved in Section 3, we generalize this result to words of weight $k$ (multisets of fixed size).  Words of weight $k$ are strings of length $l$ on an $n$ letter alphabet where the sum of the letters is $k$.

Another branch in Figure 1 follows the route from \cite{granddaddyCycles}, through \cite{cyclesAntonio} and \cite{cyclesAndre}, to our second result:  In \cite{cyclesAntonio}, the authors had shown the existence of u-cycles for chains, a particular case of the labeled posets u-cycled by the authors of \cite{cyclesAndre}.  In this paper, we show the existence of u-cycles for the set of all {\it Naturally Labeled posets on $n$ elements}.  A Naturally Labeled poset (NL poset) has nodes labeled uniquely with natural numbers.  For any two nodes $a,b$ in an NL poset, with associated labels $l_a$ and $l_b$, if $a > b$ in the poset, then $l_a > l_b$.  This result is significant because the count of these objects is not known in closed form (see OEIS entry A006455), yet we find that we can u-cycle them!

Our final result concerns $s$-ocycles of juggling patterns of length $k$ and $\le b$ balls; these objects were shown in \cite{jugglingChung} to not admit u-cycles.  Some background:  
``Site-swap" notation is a way of encoding a juggling pattern into a sequence. From a juggling perspective, each number $q$ in the juggling sequence designates a throw with height $q$. From a mathematical perspective we have two conditions to fulfill in order to have a juggling sequence. The sum of the terms of the sequence must be an integer multiple of the length of the sequence, and each term added to its position in the sequence taken mod $k$ must be distinct. The number of balls is the arithmetic mean of the sums of the terms of the sequence.  This produces an underlying permutation for each juggling sequence; for example we have that

\begin{table}[H]
\centering
\begin{tabular}{| c | c | c  |}
\hline
 Juggling Sequence & 531537 & 151140 \\
 Position & 012345 & 012345 \\
 Underlying Permutation &  543210 & 103425 \\
 \hline
\end{tabular}
\caption{Two juggling patterns converted into underlying permutations}
\end{table}
We prove an ocycle result for juggling patterns, but this result, obtained in the summer of 2013, was immediately improved and generalized by Horan \cite{horan}, who characterized those values of $s$ for which $s$-ocycles exist for juggling patterns.  Our result, though a special case of hers, is included due to the simplicity of its proof.

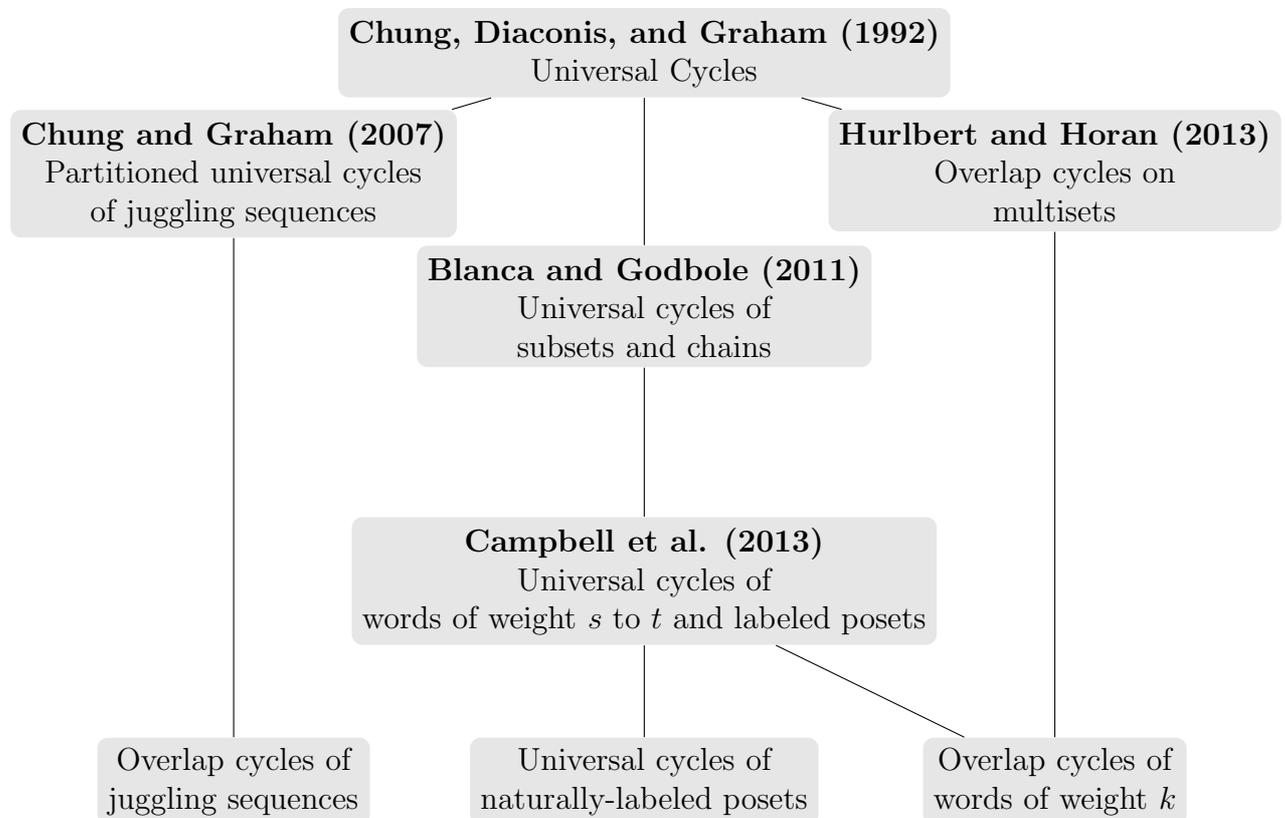
\begin{figure}[H]
 \centering
\begin{tikzpicture}[scale=.9] 
  \draw (0,0) node[fill=black!10, rounded corners](CDG)[align=center, below]{\bf{Chung, Diaconis, and Graham (1992)} \\ Universal Cycles}
  (-6, -1.5)node[fill=black!10, rounded corners](CG)[align=center, below]{\bf{Chung and Graham (2007)} \\ Partitioned universal cycles \\of juggling sequences}
  (6, -1.5) node[fill=black!10,rounded corners](HH)[align=center, below]{\bf{Hurlbert and Horan (2013)}\\Overlap cycles on \\ multisets}
  (0,-3.5)node[fill=black!10, rounded corners](AB)[align=center, below]{\bf{Blanca and Godbole (2011)} \\ Universal cycles of \\ subsets and chains}
  (0,-7.5)node[fill=black!10, rounded corners](BA)[align=center,below]{\bf{Campbell et al. (2013)} \\Universal cycles of \\ words of weight $s$ to $t$ and labeled posets}
  (0, -10.75)node[fill=black!10, rounded corners](posets)[align=center, below]{Universal cycles of \\ naturally-labeled posets}
  (6,-10.75)node[fill=black!10, rounded corners](words)[align=center, below]{Overlap cycles of \\words of weight $k$}
  (-6,-10.75)node[fill=black!10, rounded corners](juggling)[align=center, below]{Overlap cycles of \\ juggling sequences};
  \draw[-, black] (CDG) -- (HH)
  (CDG) -- (CG)
  (CDG) -- (AB)
  (BA) -- (posets)
  (AB) -- (BA)
  (BA) -- (words)
  (HH) -- (words)
  (CG)--(juggling);
\end{tikzpicture}
\caption{Relationships between Relevant Results}
\end{figure} %I can't remember if Anant said to put "our paper" above the topics we did

\section{Universal Cycles of Naturally Labeled Posets}~
A universal cycle of Naturally Labeled posets will include every NL poset of size $k$ as a $k$-window, and we progress from window-to-window in our cycle by changing one node at a time.  Specifically, we move from one poset to another by dropping the smallest node and adding a new largest node.  We then re-number the remaining nodes, preserving order.  

To show that these universal cycles exist, we follow the ``standard" process:  constructing an appropriate arc digraph, with each edge of the digraph representing a distinct NL poset of size $k$, and then showing there exists an Eulerian circuit of the digraph -- which specifies the u-cycle.  An appropriate graph has edges as previously described and vertex set $V$ consisting of NL posets of size $k-1$.  To show there exists an Eulerian circuit of such a graph, we show that the digraph is balanced, i.e., that in-degree $i(v)$ of any vertex equals its out-degree $o(v)$, and that the digraph is connected.  After demonstrating the existence of these universal cycles, we use the same method as in the preliminary {\tt arxiv} version \cite{arxiv} of Brockman, Kay and Snively \cite{graphsBrockman} to encode them as a string of integers.  

\subsection{Balancedness}~

The in-degree of any vertex in our graph is the number of ways we can append a new smallest element $s$ to the vertex NL poset.  The out-degree of any vertex is the number of ways we can append a new largest element $\ell$ to the vertex NL poset.
\begin{thm}
The number of ways to append a new extremal element to a Naturally Labeled poset, $\mu_{k-1}$, is the number of antichains, $A_1,A_2,...,A_j$ in $\mu_{k-1}$: $\vert {\cal A}\vert=\vert anti(\mu_{k-1})\vert=j.$
\end{thm}
\begin{proof}
%Let $\mu_m = {a_1, ... ,a_m}$ be an $m$-element subposet of $\lambda_n$, $m \leq n$.
%If $\mu_m$ is an antichain, we can trivially construct a new poset $\lambda_{n+1}$ by adding a node $(n+1)$ to $\lambda_n$ and setting $(n+1) > a_i$ for all $i \in [m]$.  This new $\lambda_{n+1}$ is unique to 
Let $A$ be an antichain of size $m\le k-1$ in the NL poset $\mu_{k-1}$ of size $k-1$, with $\mu_{k-1}$ corresponding to vertex $v$.  We can then declare $\ell$ to be greater than each $a\in A$ (and thus to every element $b<a$).  Thus $o(v)\ge\vert{\cal A}\vert$, where ${\cal A}$ is the set of antichains in the poset.  On the other hand if $B$ is not an antichain and we set $\ell> a$ for each $a\in B$, then the same could have been accomplished by starting with the maximal antichain $A\subseteq B$ consisting of elements with the highest labels.  Thus $o(v)=\vert{\cal A}\vert$.  A similar argument shows that $i(v)=\vert{\cal A}\vert$, and thus the digraph is balanced. \end{proof}
\begin{figure}
\centering
\begin{tikzpicture}[->,>=stealth',shorten >=1pt,auto,node distance=8cm,
                    thick,main node/.style={circle,draw,font=\sffamily\Large\bfseries}]
\node[main node] (1) {
\begin{tikzpicture}
\draw[-] (0,1) --(0,0);
\draw[fill] (0,0) circle [radius=0.1];
\node [left] at (0,0) {1};
\draw[fill] (0,1) circle [radius=0.1];
\node [left] at (0,1) {2};
\end{tikzpicture}};

\node[main node] [right of=1] (2) {
\begin{tikzpicture}
\draw[fill] (0,0) circle [radius=0.1];
\node[below] at (0,0) {1};
\draw[fill] (1,0) circle [radius=0.1];
\node[below] at (1,0) {2};
\end{tikzpicture}};

\path 
(1) edge [loop left] node [label=right:A]{
\begin{tikzpicture}
\draw[-] (-1.8,1.25)--(-1.8,-0.25);
\draw[fill] (-1.8,1.25) circle [radius=0.05];
\node [left] at (-1.8,1.25) {3};
\draw[fill] (-1.8,0.5) circle [radius=0.05];
\node [left] at (-1.8,0.5) {2};
\draw[fill] (-1.8,-0.25) circle [radius=0.05];
\node [left] at (-1.8,-0.25) {1};
\end{tikzpicture}} (1)

(2) edge [loop above] node [label=below:F]{
\begin{tikzpicture}
\draw[fill] (-1,0) circle [radius=0.05];
\node [below] at (-1,0) {1};
\draw[fill] (0,0) circle [radius=0.05];
\node [below] at (0,0) {2};
\draw[fill] (1,0) circle [radius=0.05];
\node [below] at (1,0) {3};
\end{tikzpicture}} (2)

(2) edge [loop below] node [label=G]{
\begin{tikzpicture}
\draw[-] (-1,1) --(-1,0);
\draw[fill] (-1,0) circle [radius=0.05];
\node [left] at (-1,0) {1};
\draw[fill] (0,0) circle [radius=0.05];
\node [left] at (0,0) {2};
\draw[fill] (-1,1) circle [radius=0.05];
\node [left] at (-1,1) {3};
\end{tikzpicture}} (2)

(1) edge [bend left=5] node [label=C]{
\begin{tikzpicture}
\draw[-] (-1,1) --(-1,0);
\draw[fill] (-1,0) circle [radius=0.05];
\node [left] at (-1,0) {1};
\draw[fill] (0,0) circle [radius=0.05];
\node [left] at (0,0) {3};
\draw[fill] (-1,1) circle [radius=0.05];
\node [left] at (-1,1) {2};
\end{tikzpicture}} (2)

(2) edge [bend left=5] node [label=above:~~~~~~~D]{
\begin{tikzpicture}
\draw[-] (-0.5,1) --(-1,0);
\draw[-] (-0.5,1) --(0,0);
\draw[fill] (-0.5,1) circle [radius=0.05];
\node [left] at (-0.5,1) {3};
\draw[fill] (0,0) circle [radius=0.05];
\node [left] at (0,0) {2};
\draw[fill] (-1,0) circle [radius=0.05];
\node [left] at (-1,0) {1};
\end{tikzpicture}} (1)

(2) edge [bend left=40] node [label=~~~E]{
\begin{tikzpicture}
\draw[-] (-1,1) --(-1,0);
\draw[fill] (-1,0) circle [radius=0.05];
\node [left] at (-1,0) {2};
\draw[fill] (0,0) circle [radius=0.05];
\node [left] at (0,0) {1};
\draw[fill] (-1,1) circle [radius=0.05];
\node [left] at (-1,1) {3};
\end{tikzpicture}} (1)

(1) edge [bend left=40] node [label=right:~~~~~~~~~B]{
\begin{tikzpicture}
\draw[-] (-1,1) --(-0.5,0);
\draw[-] (0,1) --(-0.5,0);
\draw[fill] (-0.5,0) circle [radius=0.05];
\node [left] at (-0.5,0) {1};
\draw[fill] (0,1) circle [radius=0.05];
\node [left] at (0,1) {2};
\draw[fill] (-1,1) circle [radius=0.05];
\node [left] at (-1,1) {3};
\end{tikzpicture}} (2)
;

\end{tikzpicture}
\caption{The arc-digraph of NL posets of size 3.  Any Eulerian circuit of this graph produces a universal cycle.}
\end{figure}
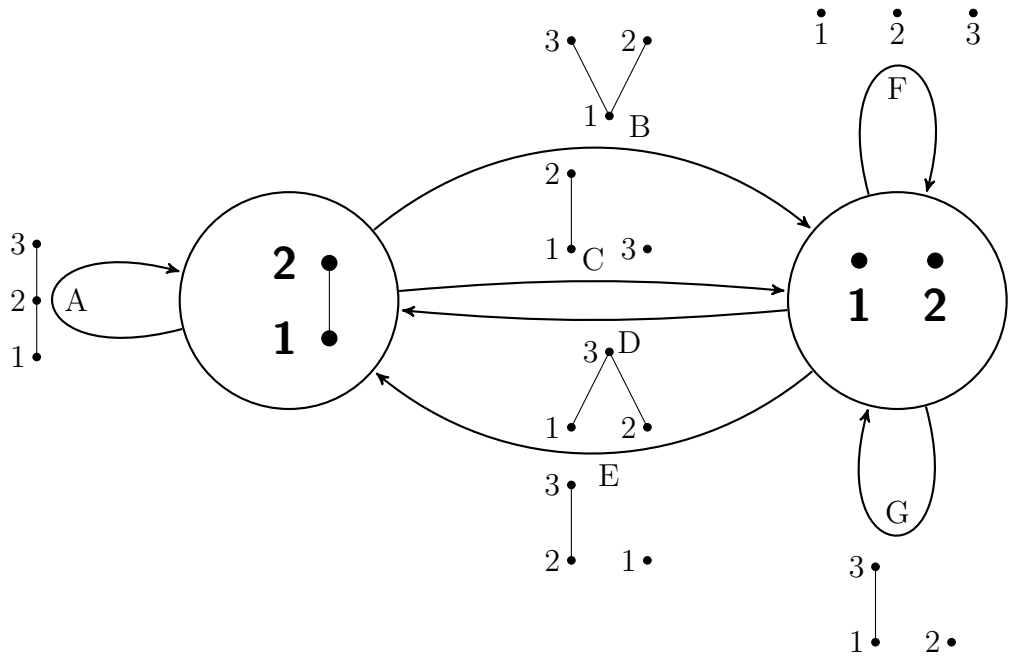

\subsection{Strong Connectedness}~
Showing strong connectedness between vertices in our graph is simple.  To move from a given vertex to a target vertex, we can take a direct path by building the NL poset of the target vertex one node at time, beginning with the smallest, which of course starts off as playing the role of the ``new largest element".  

An example of the above process in action can be seen in Figure 2, where the digraph vertices are the two NL posets of size 2, and edges are the seven NL posets of size 3.  An interesting fact may be noted:  Not only does the digraph have loops, as is quite common in these situations, but it also has multiple edges between vertices.  For example the two loops at the empty 2-poset are the result of taking the associated antichains to be $\emptyset$ or $\{1\}$.  If the seven edges in Figure 2, reading from left to right and top to bottom, are labeled $A, B, C, D, E, F,$ and $G$, then an Eulerian cycle is given, e.g., by $ABFGECD$, and this yields the ``Hasse diagram u-cycle" pictured in Figure 3.  

\begin{figure}
\centering
\begin{tikzpicture}
\tikzstyle{every node}=[draw,shape=circle,fill=black,inner sep=2pt]
\node (9) at (0,0) [label={$9=2$}] {};
\node[below right of=9] (8) [label=right:{$8=1$}] {};
\node[below left of=9] (7) [label={$7$}]{};
\draw (9) -- (8);
\draw (9) -- (7);
\node[below of=7] (6) [label=right:6]{};
\draw (7) --(6);
\node[below of=6] (4) [label=right:4]{};
\node[right of=4] (5) [label=right:5]{};
\node[left of=4] (3) [label=3]{};
\draw (6) -- (4);
\node[below of=4] (2) [label=right:2]{};
\draw (4) -- (2);
\draw (3) -- (2);
\node[below of=2] (1) [label=right:1]{};
\draw (2) -- (1);
\end{tikzpicture}
\caption{A ``Hasse diagram" u-cycle of NL posets on $\{1,2,3\}$.}
\end{figure}
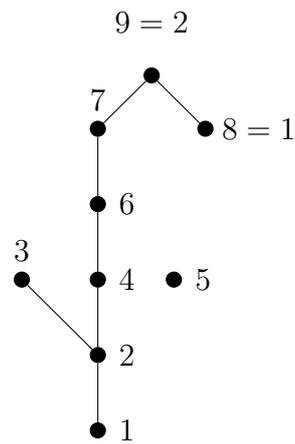

\subsection{Encoding}~

It is convenient to represent universal cycles as strings of integers, rather than as bulky diagrams as we do above.  To do so requires some careful construction, and we adapt the process in \cite{arxiv}, thinking of the Hasse diagrams of our NL posets as labeled graphs.  Since the graph u-cycles using integer codings \cite{arxiv} were not featured in the final version of the paper \cite{graphsBrockman}, we give full details here.  It is important to note that in everything related to the encoding we read strings as smallest-on-the-left and largest-on-the-right.  To begin with, we look at encoding a single NL poset of size $k$ with an integer string $(a_2,a_3,\ldots, a_k)$ of length $k-1$.  The first integer in our string corresponds to the second-smallest node in our poset, the second integer to the third-smallest node, and so on until we reach the last integer, which corresponds to the largest node -- the smallest node has no representation.

Each integer's binary representation shows the connections its node has to nodes below.  The binary representations are read from left to right, with the smallest node in the poset as the leftmost bit, and the node one step smaller than the current node as the rightmost bit.  
%If our string ends with a $4=100$, the largest node connects to only the node `$n-3$' in the Hasse diagram, while if our string ends with a $5=101$ the largest node connectes to the `$n-1$' and `$n-3$' nodes.  
We proceed with an example: the string 10455.  This string is of length 5, so it encodes an NL poset of size 6.  Consider the first integer in the string, $a_2$, which corresponds to the second-smallest node (the 2 node).  In this string $a_2 = 1$, which we convert to binary (it stays as 1), and then read as connections to nodes below the current node, with largest on the right.  This tells us the `2' node is connected to the `1' node.  Next we have $a_3 = 0=00$; this tells us the `3' node is not connected to any nodes below it.  $a_4 = 4 = 100$ (binary) tells us that the 4th node is connected to the first node.  $a_5 = 5 = 0101$ (binary), tells us that the 5th node is connected to the `4' node and the `2' node.  $a_6 = 5 = 00101$ tells us the 6th node is connected to the `5' node and the `3' node.  The NL poset this sequence encodes is given in Figure 4.

\begin{figure}
\centering
\begin{tikzpicture}
\tikzstyle{every node}=[draw,shape=circle,fill=black,inner sep=2pt]
\node (6) at (0,0) [label={6}] {};
\node[below right of=6] (5) [label={5}] {};
\node[below left of=5] (4) [label={4}] {};
\node[below left of=6] (3) [label={3}] {};
\node[below right of=5] (2) [label={2}] {};
\node[below right of=4] (1) [label={1}] {};
\draw (6) -- (5)
(6) -- (3)
(5) -- (4)
(5) -- (2)
(4) -- (1)
(2) -- (1);
\end{tikzpicture}
\caption{The NL poset encoded by the sequence 10455.}
\end{figure}
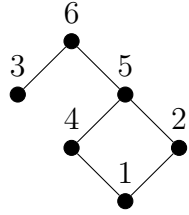
%\begin{figure}
%\centering
%\begin{tikzpicture}[->,>=stealth',shorten >=1pt,auto,node distance=8cm,
%                    thick,main node/.style={circle,draw,font=\sffamily\Large\bfseries}]
%\node[main node] (1) {
%1};
%
%\node[main node] [right of=1] (2) {
%0};
%
%\path 
%(1) edge [loop above] node {
%$11=a$} (1)
%
%(2) edge [loop above] node {
%$00=f$} (2)
%
%(2) edge [loop below] node {
%$20=g$} (2)
%
%(1) edge [bend right=10] node{
%$30=c$} (2)
%%
%(2) edge [bend left=20] node{
%$01=d$} (1)
%
%(2) edge [bend left=35] node{
%$21=e$} (1)
%
%(1) edge [bend left=5] node {
%$10=b$} (2)
%;
%
%\end{tikzpicture}
%\caption{The arc digraph that yields the encoding u-cycle for NL posets on [3].}
%\end{figure}

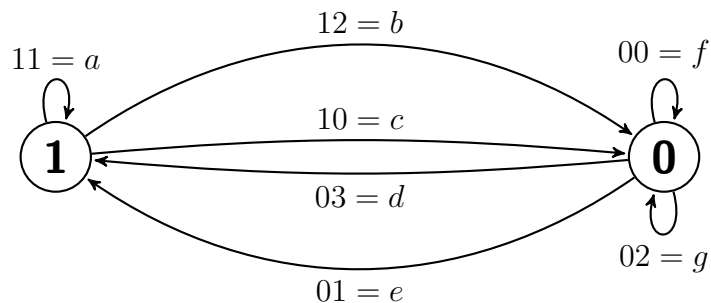
\begin{figure}
\centering
\begin{tikzpicture}[->,>=stealth',shorten >=1pt,auto,node distance=8cm,
                    thick,main node/.style={circle,draw,font=\sffamily\Large\bfseries}]
\node[main node] (1) {
1};

\node[main node] [right of=1] (2) {
0};

\path 
(1) edge [loop above] node {
$11=a$} (1)

(2) edge [loop above] node {
$00=f$} (2)

(2) edge [loop below] node {
$02=g$} (2)

(1) edge [bend left=5] node{
$10=c$} (2)

(2) edge [bend left=5] node{
$03=d$} (1)

(2) edge [bend left=35] node{
$01=e$} (1)

(1) edge [bend left=35] node {
$12=b$} (2)
;

\end{tikzpicture}
\caption{The arc digraph that yields the encoding u-cycle for NL posets on [3].}
\end{figure}

Since we can exactly encode NL posets in this way, we can directly encode our u-cycles of NL posets.  When moving through our cycle of Hasse diagrams we had ignored connections outside the $k$-window, and in our cycling of encodings we ignore binary bits which point to elements outside the $k$-window.  When we reach a $k$-window in our cycle which has entries larger than $2^j-1$, where $j$ is the position in the window (beginning from 1 on the left), we simply take those entries $a_{j+1} \mod{2^j}; j\ge 1$.  This effectively `cuts off' the connections pointing outside the $k$-window, and gives us a normal encoding string.  For instance, the encoding of the Hasse diagram u-cycle in Figures 2 and 3 will be as given in Figure 5.  

First of all, note that the corresponding upper and lower case edge labels in Figures 2 and 5 respectively are identical.  Consider for example the edge $A$, which appears as edge $a=11$ in Figure 5. Reading 11 with the smallest-on-the-left convention, we read a 1 for both the `2' and `3' nodes.  The 1 for the `2' node (same in binary), reading with largest-on-the-right means the `2' node is connected to the `1' node in our poset.  Similarly, the 1 for the `3' node tells us the `3' node is connected to the node immediately below it, i.e., the `2' node. This is the same as the poset represented by edge $A$.  Similarly the edge $d=03$ represents node 3 being connected to both nodes $1$ and $2$, and node 2 being unconnected, exactly as in edge $D$.  

The u-cycle in Figure 3 was $ABFGECD$; in Figure 5 it would be $abfgecd$, better seen as $dabfgec$ or 0{\bf 3}1{\bf 2}0{\bf 2}1, where boldface symbols represent the fact that these symbols are read mod 2 at the second instance.  A somewhat more formal exposition, using equivalence classes of window labels, may be found in \cite{arxiv}.

\subsection{Counting Universal Cycles of NL posets}
We know u-cycles of NL posets exist, but we can't say how large the u-cycle is, because there is no method to count NL posets.  It would be of interest, then, to examine the graph produced in our u-cycle method, and to attempt to count the number of u-cycles of NL posets of size $k$.  To do so, we'll use the BEST theorem of  de Bruijn, van Aardenne-Ehrenfest, Smith and Tutte \cite{best1}, \cite{best2} -- which states that the number of Eulerian circuits in an Eulerian digraph $G$ equals
\begin{equation}ec(G)=t_w(G)\prod_{v\in V}(deg(v)-1)!,\end{equation}
where $t_w(G)$ is the number of trees directed at any one $w\in G$ (arborescences).  Now it is well known (from the digraph version of the matrix tree theorem) that $t_w(G)$ is given by any cofactor of the Laplacian of the arc digraph $G$.  Define $anti(\lambda)$ to be the number of antichains in poset $\lambda$, and $NL(k)$ to be the number of NL posets of size $k$.  Throughout, $\mu$ will refer to an NL poset of size $k-1$, and $\lambda$ to an NL  poset of size $k-2$.

Let us consider the degree matrix and adjacency matrix for $G$.  Since the  graph has vertices corresponding to NL posets of size $k-1$, both these matrices will be square matrices of size $NL(k-1) \times NL(k-1)$.  The degree matrix is fairly straightforward -- it will contain, along the diagonal, $anti(\mu_i)$ with $i \in [NL(k-1)]$.  This is due to the degree of each vertex being determined by the number of antichains in the corresponding NL poset.  The ordering of these $anti(\mu_i)$ values is arbitrary but can be made somewhat more ``block-diagonal"-like based on considerations in the adjacency matrix.
%-- we will later define our ordering based on considerations in the adjacency matrix.

The adjacency matrix is a little more interesting.  The adjacency of any $\mu_i$ is determined by the $k-2$ largest elements in $\mu_i$, because the same arrangement of elements must appear as the $k-2$ smallest elements of an adjacent $\mu_j$.  We can consider these adjacency-determining elements as their own NL poset of size $k-2$, denoted by $\lambda_i$.  This $\lambda_i$ appears as the smallest $k-2$ elements in adjacent posets, so we'll give it the corresponding name $\lambda'_j$, and say that poset $\mu_i$ is adjacent to poset $\mu_j$ if $\lambda_i = \lambda'_j$.  So in our adjacency matrix, the column corresponding to $\mu_i$ will have a nonzero entry for every $\mu_j$ such that $\lambda_i = \lambda'_j$.  Since there are $anti(\lambda'_j)$ such NL posets, we will have $anti(\lambda_i)$ nonzero entries in the $\mu_i$ column.  Furthermore, since there are $anti(\lambda_i)$ posets similar to $\mu_i$, that is, posets which contain $\lambda_i$ as their largest $n-2$ elements, we will have $anti(\lambda_i)$ columns which each contain $anti(\lambda_i)$ nonzero entries.

This gives us the interesting observation (equivalent to the fact that in a digraph, the sum of the degrees equals the size) that:
\begin{equation}\sum\limits_{i=1}^{NL(k-2)} anti(\lambda_i)  = NL(k-1).
\end{equation}
That is, the number of NL posets of size $k-1$ is equal to the sum of the number of antichains of every NL poset of size $k-2$.   We show next that these considerations can lead to some progress in the determination of the number of Eulerian circuits in $G$.  We first consider small cases.  For $k=3$, it is east to see from Figure 2 that the degree and adjacency matrices $D$ and $A$ are given by 
\[ D=\left( \begin{array}{cc}
3 & 0\\ 
0 & 4\end{array} \right); A=\left( \begin{array}{cc}
1 & 2\\ 
2 & 2\end{array} \right),\]
so that the Laplacian is
\[ L=\left( \begin{array}{cc}
2 & -2\\ 
-2 & 2\end{array} \right),\]
and thus $t_w(G)=2$ and, by (1), $ec(G)=2\cdot2!\cdot3!=24$, a fact that is readily verified by simple counting arguments on the arc digraph in Figure 2.  For $k=4$ we label the vertices of $G$ using the edge labels $A$ through $G$ as in Figure 2 to get the degree vector ${\bf d}=(4\ 5\ 6\ 5\ 6\ 8\ 6)$, adjacency matrix
\[ A=\left( \begin{array}{ccccccc}
1 & 1 & 2 & 0 & 0 & 0 & 0 \\
0 & 0 & 0 & 1 & 1 & 2 & 1 \\
0 & 0 & 0 & 1 & 2 & 2 & 1 \\
1 & 2 & 2 & 0 & 0 & 0 & 0 \\
2 & 2 & 2 & 0 & 0 & 0 & 0 \\
0 & 0 & 0 & 2 & 2 & 2 & 2 \\
0 & 0 & 0 & 1 & 1 & 2 & 2
\end{array} \right),\]
and Laplacian
\[ L=\left( \begin{array}{ccccccc}
3 & -1 & -2 & 0 & 0 & 0 & 0 \\
0 & 5 & 0 & -1 & -1 & -2 & -1 \\
0 & 0 & 6 & -1 & -2 & -2 & -1 \\
-1 & -2 & -2 & 5 & 0 & 0 & 0 \\
-2 & -2 & -2 & 0 & 6 & 0 & 0 \\
0 & 0 & 0 & -2 & -2 & 6 & -2 \\
0 & 0 & 0 & -1 & -1 & -2 & 4
\end{array} \right),\]
whose leading cofactor is 4900 -- which leads to $4900\cdot6\cdot24\cdot120\cdot24\cdot120\cdot5040\cdot120=$ 147,483,721,728,000,000 Eulerian paths.  The sequence 
$ (24, 147483721728000000,\ldots)$ is not to be found in OEIS.

We next turn to asymptotic considerations, using (2) to gain some handle on the product term in (1) (of course, estimating the number of arborescences is a separate matter altogether).  We have by (2) and Stirling's approximation
\begin{eqnarray*}
\prod_{v\in V}(deg(v)-1)!&=&\exp\lc\sum_v\ln(deg(v)-1)!\rc\\
&=&\exp\lc\sum_v\ln(deg(v))!-\ln(deg(v))\rc\\
&\sim&\exp\lc\sum_v\lr\ln\lc{\sqrt{2\pi deg(v)}}\lr\frac{deg(v)}{e}\rr^{deg(v)}\rc-\ln(deg(v))\rr\rc\\
&=&\exp\lc\sum_vdeg(v)\ln(deg(v))-deg(v)-\frac{1}{2}\ln(deg(v))+\ln{\sqrt{2\pi}}\rc\\
&=&\exp\lc\sum_vdeg(v)\ln(deg(v))\{1+o(1)\}\rc.
\end{eqnarray*}
Now since $\sum_vdeg(v)=NL(k)$, we have that $\sum_vdeg(v)\ln(deg(v))\sim NL(k)\ln NL(k)$, and so
\[
\prod_v(deg(v)-1)!\sim NL(k)^{NL(k)}.
\]

\section{Overlap Cycles of Words of Weight $k$}~

As indicated in the Introduction, the work of \cite{hh2} reveals that we may use the characteristic vector coding to form $s$-ocycles of the $k$-subsets of $[n]$ provided that $gcd(s,n)=1$.  We want to extend this result, and show that there exists an $s$-overlap cycle for words of weight $k$ and length $n$ on the $q+1$-letter alphabet $\{0,1,\ldots,q\}$, provided that $s\in[n-2]$,  $gcd(n, s) = 1$, and $q \leq k$. The obvious correspondence between words and multisets thus yields $s$-ocycles for $k$-multisets of an $n$ element set in which no element may appear more than $q$ times.  The results in \cite{hh2} on the ocyclability of all permutations of a fixed multiset show that words of weight $k$ with a fixed composition {\it can} be ocycled; we seek to do this for words of weight $k$ with {\it any} composition.  The coding that we use for multisets of $[n]$ is thus of length $n$, with the $i$th element in a string indicating how many times element $i$ appears in the multiset.

\subsection{Balancedness}~

Consider a graph with vertices representing overlaps (strings of length $s$ and weight $\le k$ that can be extended to an edge of weight $k$) and edges representing weight $k$ words of length $n$.  Observe that for any vertex $ v = v_1 v_2 ... v_s$ and for any incoming edge $w = w_1 w_2 ... w_{n-s} v_1 v_2 ... v_s$ we have an analogous outgoing edge $w' = v_1 v_2 ... v_s w_1 w_2 ... w_{n-s}$. Thus, in-degree is equal to out-degree for every vertex $v$. 

\subsection{Weak Connectedness}~

Now, given some vertex $v = v_1 v_2 ... v_s$  with $ \sum_{i=1}^{s} v_i = k_v\le k$ where $k_v$ is the weight of $v$, we show that we can reach the ``most-maximal" vertex $m = m_1 m_2 ... m_s$ with 
\begin{align*} m_{s- \lfloor k/q \rfloor+1} &= ... = m_s = q,  \end{align*}
 \begin{align*}m_i = 0 \text{ for $1\le i\le s- \lfloor k/q \rfloor $.} \end{align*}

 The most maximal vertex has as many characters as possible with weight equal to that of ``the largest letter of the alphabet" with the remaining terms equaling zero.
 
 We first append letters to the end of $v$ to create a word $w = w_1 w_2 ... w_n$, of required weight $k$, with $w_1 w_2 ... w_s = v_1 v_2 .. v_s$.  For $\lfloor \frac{k-k_v}{q} \rfloor$ letters $w_j$ such that $s<j\leq n$, we set $w_j = q$ and if $\frac{k-k_v}{q} \notin \mathbb{Z}$ then for one such $w_j$, we set $w_j = (k-k_v) \mod{q}$. For all other $w_j$, let $w_j=0$. We then rearrange the letters of $w$, placing $w_i$ such that $0 \le w_i < q$ in descending order from the beginning with $w_i =q$ at the very end, to make a new word $w' = w'_1 w'_2 ... w'_n$. The fact that $gcd(n,s)=1$ permits us to do this; we just use Lemma 4.2 in \cite{hh2} to follow the Eulerian cycle that exists for ``permutations of fixed multisets" until we reach the multiset with the special ordering we seek.

This word $w'$ points towards  a vertex $v' = v'_1 v'_2 ... v'_s$ with $w_{n-s+1}' w_{n-s+2}' ... w_n' = v'_1 v'_2 ... v'_s$.  We construct a new word $w''$ using the append and rearrange method described above.  We continue this process, alternating between words and vertices until we reach the most-maximal vertex. The above algorithm is guaranteed to terminate since the rearrangement leads to a decrease in the weight of vertices (except for the $q$s at the end) as well as an increase in the number of zeros. 

For example, let $n=9, k=15, s=7, q=9$. We start at $v = 1332051$.  $k_v = 15$, so we append 0s to make $w = 133205100$.  Then we rearrange as described above:  $w' = 533211000$.  $w'$ points towards $v' = 3211000$, and we repeat the process with this new vertex.  $v' = 3211000 \rightarrow w''' = 321100080 \rightarrow w^{(4)} = 832110000 \rightarrow v'' = 2110000 \rightarrow w^{(5)} = 211000092 \rightarrow w^{(6)} = 221100009 \rightarrow v''' = 1100009 \rightarrow w^{(7)} = 110000940 \rightarrow w^{(8)} = 411000009 \rightarrow v^{(4)} = 1000009 \rightarrow w^{(9)} = 100000950 \rightarrow w^{(10)} = 510000009 \rightarrow v^{(5)} = m = 0000009$.

~

Since in-degree is equal to out-degree, and the graph is weakly connected, the graph is eulerian and we have an $s$-ocycle.

\section{Overlap Cycles of Juggling Sequences}~

Chung and Graham (2007) produced u-cycles of site-swap juggling patterns \cite{jugglingChung}, but the u-cycles didn't work perfectly -- they were split up into unions of disjoint cycles. However, using $s$-ocycles, we demonstrate that one can cycle through all juggling patterns of length $n$ and number of balls $\leq b$. This result has been proved independently and in greater generality by Horan \cite{horan}, but our proof in the special case appears to be somewhat simpler.  We will assume in this section that $\gcd(n,n-2)=1$, i.e., we take $s=n-2$.

\subsection{Cyclic Shifts}~

Given an arbitrary juggling pattern edge $j=j_1 j_2 ... j_n$ with underlying permutation $p=p_1 p_2 ... p_n $ such that  $p_1 \pmod{n} \not \equiv p_2 \pmod{n} \not \equiv ... \not \equiv  p_n \pmod{n}$, and pointing towards the vertex $j=j_{n-s+1}\ldots j_n$ we construct an edge leaving from that vertex and show that it is a legal juggling sequence.  Specifically, we drop the first $n-s$ letters of the edge, and place them in order at the end of the outgoing edge.  
Thus, we move from $j$ to the vertex $v = j_{n-s+1} j_{n-s+2} ... j_n$ and then to the outoing edge $j' = j_{n-s+1} j_{n-s+2} ... j_n j_1 j_2 ... j_{n-s}$.  We now show that the outgoing edge is a legal juggling sequence, by showing the underlying pattern consists of distinct elements modulo $n$.  The characters $j_{n-s+1} j_{n-s+2} ... j_n$ move backwards in the string $s$ spaces, so each corresponding $p_{n-s+1} p_{n-s+2} ... p_n$ has $s$ subtracted from it modulo $n$.  The characters $j_1 j_2 ... j_{n-s}$ move forward in the string $n-s$ spaces so each corresponding $p_1 p_2 ... p_{n-s}$ has $n-s$ added to it modulo $n$.  Since both of these operations are modulo $n$ and $n-s \equiv -s \pmod{n}$, we modify each term of the underlying permutation the same way, and end up with a new legal juggling sequence.    Since we can apply this operation to any edge in the graph, we can repeat it to obtain any cyclic shift of a given juggling pattern.

\subsection{In Degree and Out Degree}

We establish a bijection between incoming and outgoing edges for a vertex $v$ by means of cyclic shifts. If we have an incoming edge $j$, then we can shift by overlap size to obtain an outgoing edge $j'$ from $v$. Similarly, if we have an outgoing edge $r$, we can shift by overlap size to obtain an incoming edge $r'$ from $v$.  There is a bijection between $j$ and $j'$ (or $r$ and $r'$) and we are done.

\subsection{Weak Connectedness}~

We now show weak connectedness for $s=n-2$. Given any juggling sequence of length $n$ and number of balls $\leq b$, we'll show there exists a path to the vertex consisting of all 0s.  

Given some juggling sequence $j = j_1 j_2 ... j_n$ we begin by cycling so that the character (or one of the characters) with the largest weight is in the $j_2$ position.  This will be the default step.  At any point in the process, if a character is of weight $n$, we reduce it to $0$. Reinitialize by renaming the new edge $j= j_1 j_2 ... j_n$.  If $j_1 \neq 0$, we let $m_2 = j_2+1$ and $m_1 = j_1 - 1$, and transition to the edge $j_3\ldots j_{n}m_2m_1$.  If $j_1 = 0$, we cycle until the character with the next largest weight is in the $j_1$ position.  If $j_1 \neq 0$, we let $m_2 = j_2 + 1$ and $m_1 = j_1 - 1$, effectively breaking down this next largest element.   Again, after reinitialization, we cycle so that the largest element is in the $j_2$ position and proceed based on the value of $j_1$.  
The process is continued until we have reached the vertex that consists of all $0$'s.

For example, let $n=9$ and $s=7$. We want to show there exists a path to the sink vertex of all $0$s. We start at $300300300 \rightarrow 030030030 \rightarrow 300300300 \rightarrow 030030012 \rightarrow 203003001 \rightarrow 300300111 \rightarrow 130030011 \rightarrow 003001140 \rightarrow 140003001 \rightarrow 000300150 \rightarrow 150000300 \rightarrow 000030060 \rightarrow 060000030 \rightarrow 300600000 \rightarrow 060000012 \rightarrow 206000001 \rightarrow 600000111 \rightarrow 160000011 \rightarrow 000001170 \rightarrow 170000001 \rightarrow 000000180 \rightarrow 180000000 \rightarrow 000000000$.

It is conceivable that a clever adaptation of this process will work for $s\ne n-2$, by introducing new symbols $m_1,\ldots,m_{n-s}$, but we do not explore this idea, particularly in light of Horan's \cite{horan} general result.

     \end{document}